\documentclass[11pt]{article}
\usepackage{amssymb, amsmath, amsthm, color, xcolor, amsfonts,tabularx, authblk, relsize, geometry}
\usepackage[english]{babel}
\def\supp{\mathop{\rm supp}\nolimits}

\newtheorem{theorem}{Theorem}[section]

\newtheorem{lemma}[theorem]{Lemma}

\newtheorem{corollary}[theorem]{Corollary}
\newtheorem{definition}[theorem]{Definition}
\newtheorem{remark}[theorem]{Remark}
\newtheorem{example}[theorem]{Example}
\newtheorem{notation}[theorem]{Notations}

\newcommand{\R}{\mathbb{R}}

\renewenvironment{proof}[1][.]{%
\bigskip\noindent{\bf Proof#1 }}{%
\hfill$\blacksquare$\bigskip}

\newcommand{\mP}{\mathcal P}

\newcommand{\mcal}[1]{\mathcal{#1}}
\begin{document}
\pagestyle{myheadings}
\title{The Hutchinson-Barnsley theory for iterated function systems with general measures}
\author[1]{Elismar R. Oliveira
\thanks{E-mail: elismar.oliveira@ufrgs.br}}
\author[2]{Rafael R. Souza
\thanks{E-mail: rafars@mat.ufrgs.br}}
\affil[1]{Universidade Federal do Rio Grande do Sul, Universidade de Aveiro}
\affil[2]{Universidade Federal do Rio Grande do Sul}

\date{\today}
\maketitle

\begin{abstract}
In this work we present iterated function systems with general measures(IFSm) formed by a set of maps $\tau_{\lambda}$ acting over a compact space $X$,  for a compact space of indices, $\Lambda$. The Markov process $Z_k$ associated to the IFS iteration is defined using a general family of probabilities measures $q_x$ on $\Lambda$, where $x \in X$:  $Z_{k+1}$ is given by $\tau_{\lambda}(Z_k)$, with $\lambda$ randomly chosen according to $q_x$. We prove the existence of the topological attractor and the existence of the invariant attracting measure for the Markov Process. We also prove that the support of the invariant measure is given by the attractor and results on the stochastic stability of the invariant measures, with respect to changes in the family $q_x$. 
\end{abstract}
\vspace {0.8cm}





\section{Introduction}

An iterated function system (IFS for short) is a family of maps acting from a set to itself, and, among several applications, such as in the so called fractal geometry,   can be seen both as a generalization of a usual dynamical system, as well as a model to define Markov Probabilistic Processes.

We begin by observing that a IFS can be seen as a generalization of a dynamical system.
Remember that a dynamical system is defined by a map $T: X \to X$ acting from a set $X$ to itself.  An initial point $x_0 \in X$ can be  iterated by $T$ producing the orbit $\{x_0, T(x_0), T^2(x_0), ... \}$, whose limit or cluster points are the objects of main interest, from a dynamical point of view. On the other hand,  in a broader sense, an IFS, denoted by $\mathcal{R}=(X, \tau)$, is a family of maps $\tau_\lambda$ acting from $X$ to itself, indexed by some index set $\lambda \in \Lambda$.

Given an initial point $x_0$, it can be iterated by choosing at each step a different map $\tau_{\lambda}$, producing multiple orbits $\{Z_j, j \geq 0\}=\{Z_0=x_0, Z_1=\tau_{\lambda_0}(x_0), Z_2= \tau_{\lambda_1}(\tau_{\lambda_0}(x_0)), ... \}$. We no longer have a single orbit as we had in the case of dynamical systems, now the orbit is in fact  a set of orbits controlled by a sequence $(\lambda_0, \lambda_1, ...) \in \Lambda^{\mathbb{N}}$ that has to be chosen according to some (usually probabilistic) rule.

Before getting in the probabilistic aspects of the theory,  we can search for possible topological limit or cluster points and its geometrical features. Suppose that $(X,d)$ is a complete metric space, each $\tau_{\lambda}$ is continuous in both variables. The Hutchinson-Barnsley (or fractal) operator $F: K(X) \to K(X)$, is defined  by,
$$F(B)=\bigcup_{\lambda \in \Lambda}\tau_{\lambda}(B),$$
for $B \in K(X)$, where $K(X)$ is the family of nonempty compact sets of $X$.

We call a compact set  invariant or fractal if $F(\Omega)=\Omega$. Additionally $\Omega$ is called a fractal attractor if the orbit of $B$ by $F$, given by $\{B, F(B), F^2(B), ... \}$ converge, w.r.t. the Hausdorff-Pompeiu metric to $\Omega$, for any $B \in K(X)$ (see \cite{berinde2013role} for details on the Hausdorff-Pompeiu metric).

Now, passing to the probabilistic aspects of the theory,  let $\mathcal{P}(X)$ be the set of Borel probability measures defined on the metric space $X$, and  $q=\{q_x \}_{x \in X}$ be a family of probability measures on $\mathcal{P}(X)$ indexed by $x \in X$. We denote $\mathcal{R}=(X, \tau, q)$ an IFS with measures (IFSm). The particular case where $q_x=cte, \forall x \in X$ is well known in the literature as IFS with probabilities (IFSp). We have multiple cases of study in the literature concerning the nature of the index set $\Lambda$ (i.e. the quantity of maps defining the IFS as finite, compact, countable, etc), and $q$ (most notably if $q$ is a singleton or if $q_x$ actually depends on $x$).

We must emphasize that most of the difficulties in this work arises from the nature of the dependence $x \to q_x$. This degree of generality allow us to modeling situations where the probability transitions are non uniform with respect the phase space. Only a few references in the literature have addressed this class of problems, mostly in the context of a finite set of maps, as it presents substantial challenges to the classical constructions that have proven effective for IFSp problems.  We believe that our approach not only establish the existence of a new object, the invariant measure for an IFS with measures under Lipschitz contractivity of maps and weights, but also provides several formulations under which this result is valid such as eventually contractive mappings, weakly hyperbolic mappings (diameter of dynamical balls decreasing to zero) and average contractive weights. This opens a wide range of applications. For instance, the Ruelle-Perron-Frobenius theorem and the Ergodic theorem are not know in this setting.

We consider the Markov process $\{Z_{0}, Z_{1}, Z_{2}, ...\}$ where $Z_0$ is randomly chosen according to an initial distribution $\mu_0 \in \mathcal{P}(X)$, and for each $j \geq 0$, $Z_{j+1}\in X$ is
obtained from the previous $Z_{j}$ by defining
$$Z_{j+1}=\tau_{\lambda_j}(Z_j)$$ where the index $\lambda_j$ is randomly chosen according to the probability distribution $q_{X_j}$.

If we call $\mu_j$  the probability distribution of $Z_j$,
then the distribution of  $Z_{j+1}$, for $j \geq1$,  is given by the recurrence relation $\mu_{j+1}= \mcal{L}_{  q}(\mu_{j})$, where
$\mcal{L}_q:\mathcal{P}(X) \to \mathcal{P}(X)$ is the Markov operator, defined for any $\nu \in \mathcal{P}(X)$ and $f \in C(X,\R)$,
by $$\int_{X} f(x) d\mcal{L}_{  q}(\nu)(x)=  \int_{X} B_q(f)(x) d\nu(x),$$
and $B_q: C(X,\R) \to C(X,\R)$ is the transfer operator,
 given by
\begin{equation*}
	B_{q}(f)(x)=\int_{\Lambda} f(\tau_{\lambda}(x)) dq_x(\lambda).
\end{equation*}
(We denote by $C(X,\R)$ the Banach space of all real continuous functions equipped with supremum norm $\| \; \|_{\infty}$.)
Therefore, a fixed point for $\mcal{L}_q$, if exists, represents the invariant measure, or stationary distribution of the Markov process $\{Z_{j},\; j \geq 0\}$.
An attracting fixed point of $\mcal{L}_q$, if exists, is called the attracting invariant measure for the process $\{Z_{j},\; j \geq 0\}$.
(Here we mean by attracting fixed point a probability $\mu_e$ such that, for any $\mu \in \mathcal{P}$, we have that  $\mcal{L}^k_{  q}(\mu)$ converges to $\mu_e$, when $k \to \infty$,  w.r.t. the
Monge-Kantorovich  distance.)
See \cite{Doob48}, \cite{Lasota2002} and\cite{BarnDemko85}, for details on Markov operators and its connection with IFS, and Ruelle~\cite{ruelle1967variational,ruelle1968statistical}, Walters~\cite{walters1975ruelle} and Fan~\cite{fan1999iterated})
for the Transfer operator, also called Ruelle operator.
See \cite{hutchinson1981fractals} for details on the Monge-Kantorovich metric.

A remarkable fact that shows how the dynamical aspects of the theory
are related to the statistical aspects is the fact that, under appropriate conditions on the maps $\tau_{\lambda}$ and  also on the map $x \to q_x$, the support of the invariant measure coincides with the fractal attractor set, i.e., the attracting fixed point of the Hutchinson-Barnsley operator.

The study of the conditions under which we have, for a given IFS, a fractal attractor which is the support of an invariant measure is called the Hutchinson-Barnsley theory.

Historically, the first papers on IFSp, by Hutchinson, Barnsley and many others in the beginnings of the 80's
(see for example Hutchinson~\cite{hutchinson1981fractals}),
 considered $\Lambda=\{1,2,..., n\}$ as a finite set and $q$ a singleton (i.e. a probability $q_x=q$ that is unique).
In these papers we have  a finite  number of maps, and each one is chosen according to a fixed probability $p_j>0$ where $p_1+ \cdots + p_n=1$.
Under these conditions the transfer operator is given by
\begin{equation}\label{eq:class_transf_op}
	B_{q}(f)(x)=\int_{\Lambda} f(\tau_{\lambda}(x)) dq(\lambda)= \sum_{j=1}^{n} p_j \, f(\tau_{j}(x)),
\end{equation}
for any $f \in C(X,\R)$, where the measure $q$ is given by $$dq(\lambda)= \sum_{j=1}^{n} p_j  \delta_{j}(\lambda), \; \forall x \in X.$$

A first generalization, yet for a finite $\Lambda=\{1,2,..., n\}$, was for IFSp where the constant probability $p_j>0$ such that $p_1+ \cdots + p_n=1$ is replaced by a variable probability $p_j(x)>0$ where $p_1(x)+ \cdots + p_n(x)=1$ for all $x \in X$. Now, the transfer operator is defined by
\begin{equation}\label{eq:class_variable_transf_op}
	B_{q}(f)(x)=\int_{\Lambda} f(\tau_{\lambda}(x)) dq_{x}(\lambda)= \sum_{j=1}^{n} p_j(x) \, f(\tau_{j}(x)),
\end{equation}
where the measure $q_{x}$ is given by $dq_{x}(\lambda)= \sum_{j=1}^{n} p_j(x)  \delta_{j}(\lambda), \; \forall x \in X$,  for any $f \in C(X,\R)$. Very general conditions for the existence of the invariant measure for such IFS are given in \cite{barnsley1988invariant}.
See also \cite{fan1999iterated}.

We would like to mention that  the compactness assumption on $\Lambda$ is essential for our techniques to work. There is a whole area dedicated to the case where $\Lambda$ is given by a countable set, see the pioneering work of   Urba\'{n}ski~\cite{Urb96} on families of conformal contractions and similar papers from the following years for details.

The case where $\Lambda$ is a more general space (a compact or even measurable set) was considered by Stenflo~\cite{stenflo2002uniqueness}, Arbieto~\cite{Arbieto2016WeaklyHyperbolic} and Mengue~\cite{MO}.

Stenflo~\cite{stenflo2002uniqueness}
considered a IFS $(X, \tau, \mu)$ for a fixed $\mu$ over an arbitrary measurable index space $\Lambda$. He uses a  slightly different approach, supposing that the iterations from $Z_0 \in X$ are $Z_{j+1}=\tau_{I_j}(Z_{j})$, controlled by a sequence of i.i.d variables $\{I_j \in \Lambda\}_{j \in \mathbb{N}}$, with a common distribution $\mu$,
which essentially means we have a IFSm  with a single probability $q_x=\mu, \forall x \in X$.
The main goal of Stenflo~\cite{stenflo2002uniqueness} is to establish, when $B_{q}$ is Feller, the existence of  an unique attracting invariant measure $\pi$.
Other interesting results,  including some results on stochastic stability (see section \ref{sec:eqmeasdependscontinuously}), for the somewhat simpler case where the $\Lambda$ is countable, are contained in Mendivil~\cite{mendivil1998generalization}, but also in the case we have a single probability  $q=\mu$.

Arbieto~\cite{Arbieto2016WeaklyHyperbolic} also consider a IFS $(X, \tau, \mu)$ for a compact space $\Lambda$, and again supposing a unique probability $q=\mu$, proves the existence of the  topological attractor, as well as the existence of invariant measures, under what he calls the weakly hyperbolic hypothesis on the IFS (see also \cite{Edalat}). He also proves an interesting result, concerning an ergodic theorem for the IFS, generalizing previous results by Elton~\cite{elton1987ergodic}, and also deals with the Chaos Game Theorem (see for example \cite{Arbieto2016WeaklyHyperbolic}).

Mengue~\cite{MO} considered two uniformly contractive IFSs over compact metric spaces with a compact sets of indices and established the duality for holonomic measures maximizing a cost function $C$. The main improvement there was to consider a family of measures of the form $q_x=e^{C(\lambda, x)}d\alpha(\lambda)$ becoming a true IFSm. A similar construction was already known from Lopes~\cite{lopes2015entropy}, for the particular case where the IFS is formed by the inverse branches of the shift map on the product of compact metric spaces.

We also mention the recent work of Mierlus–Mazilu~\cite{MN24} which studied IFS with measures $(T,\omega, q)$ in a compact space of parameters $[a,b] \in \mathbb{R}$, $\omega_\alpha: T \to T, \; \alpha \in [a,b]$, where $(T,d)$ is a compact metric space. It is assumed that each $\omega_\alpha$ is a $r_\alpha$-Lipschitz contraction and $\sup_{\alpha\in [a,b]} r_\alpha <1$. In this case, the measures $q$ are in the form $q_x=p(\alpha)d\nu(\alpha)$ (not depending on $x\in T$). Under this hypothesis they proved that there exists the attractor, the invariant measure and its support  is the attractor. After, they do analyze the behavior of sequences for theses IFSs, a kind of stability result.  On one hand no continuity with respect to $\alpha$ is required, however the measures and  spaces involved are very particular.

Note that, in the previous references Stenflo, Mendivil, and in Arbieto, the parameter space $\Lambda$ is more general, but  the probability used to define which map is used to iterate the process is always fixed (independent of $x \in X$).

Our goal in this paper is to establish the Hutchinson-Barnsley theory
and also some results on stochastic stability in the case the parameter space $\Lambda$ is a compact metric space, and the family of probabilities $q=\{q_x\}_{x \in X}$ admits different probabilities for each $x \in X$.
If $\mathcal{R}=(X, \tau, q)$ is an IFSm,  $F_{\mathcal{R}}$ is the Hutchinson-Barnsley operator, and
$\mcal{L}_{q}$ is the Markov operator, then
under mild assumptions we will prove there exists a unique compact set $A_{\mathcal{R}}$ called fractal attractor and an unique probability $\mu_{\mathcal{R}}$ called invariant or Hutchinson measure, such that
\begin{enumerate}
	\item $F_{\mathcal{R}}(A_{\mathcal{R}})=A_{\mathcal{R}}$ and for any nonempty compact set $B \subseteq X$ we have $F_{\mathcal{R}}^{k}(B) \to A_{\mathcal{R}}$, w.r.t. the Hausdorff distance;
	\item $\mcal{L}_{  q}(\mu_{\mathcal{R}})=\mu_{\mathcal{R}}$ and for any probability $\nu \subseteq \mathcal{P}(X)$ we have $\mcal{L}_{  q}^{k}(\nu) \to \mu_{\mathcal{R}}$, w.r.t. the Monge-Kantorovich distance;
	\item $\supp(\mu_{\mathcal{R}})=A_{\mathcal{R}}$.
\end{enumerate}
We will also prove theorem \ref{thm: continuous dependence measure} which shows the stochastic stability of the invariant measure with respect to perturbations of the family $Q$:
consider, for each $n \in \mathbb{N}$, an IFSp  $\mathcal{R}^{(n)} :=(X, \tau, q^{(n)})$  satisfying appropriate conditions.
Additionally assume the sequence $ q^{(n)}_{x} \to  q^{*}_{x}$ with respect to the distance $d_{MK}$ in $\mathcal{P}(\Lambda)$, uniformly in $x \in X$.
Let $\mu^{(n)}$ be the unique invariant measure for $\mathcal{R}^{(n)}$. Then we will prove that there exists $\mu^{*}$ a unique invariant measure for $\mathcal{R}^{*} :=(X, \tau, q^{*})$, and $\mu^{(n)}  \to  \mu^{*}$ with respect to the distance $d_{MK}$ in $\mathcal{P}(X)$.

It is worth to mention a previous experience with IFSm,
 in \cite{Brasil2022}, where the parameter space $\Lambda$ is a compact metric space, and the family of probabilities $q=\{q_x\}_{x \in X}$ admits different probabilities. However, in \cite{Brasil2022} we considered only the thermodynamical formalism for IFSm. We had not worked on the Hutchinson-Barnsley theory for these IFSm. The goal in \cite{Brasil2022} was to study the thermodynamical formalism for the IFSm, which comprises good definitions for transfer operators, invariant measures, entropy, pressure, equilibrium measures, and a  variational principle. We also used  these tools to characterize the solutions of the ergodic optimization problem. Other approaches for IFS thermodynamic formalism were developed by \cite{simon2001invariant,mauldin2000parabolic,hanus2002thermodynamic,lopes2009entropy,kaenmaki2004natural, lopes2015entropy,ElismarJairo24}  and many others. See the introduction of \cite{Brasil2022} for a more detailed description.

The structure of the paper is the following:
In Section~\ref{sec:IFSs on compact spaces with a compact set of maps} we present the basic definitions on IFS with probabilities, including the Monge-Kantorovich distance, the transfer and Markov operators, the fractal attractor and invariant measures, and state the main results of the papers which relates all these objects.
In Section~\ref{sec:statpointofview} we present the statistical point of view, where we present the Markov process defined by the IFSm, and prove that the Markov operator indeed gives the distribution of such process, proving that its fixed point, the invariant measure, is in fact the stationary distribution of the Markov Process.
In Section~\ref{sec:hyp} we present the hypotheses under which our results will be proved, and present some examples that fit our hypotheses. We will work thorough the text with essentially two sets of hypotheses, the second more technical but with more applications. In fact the second set implies the first, and we choose to keep both sets here, as the proofs are less technical in the first set of hypotheses.
 In Section~\ref{sec:existence-of-the-attractor} we have the same situation: the existence of the fractal attractor can be proved under these two sets of hypotheses.
In Section~\ref{sec:Existence of the invariant measure} we prove the existence of the invariant measure under the two different sets of hypothesis, the simpler set which provides a more direct proof and the more technical set, where we assume weaker hypotheses. In Section~\ref{sec:support} we prove that the support of the invariant measure coincides with the fractal attractor, under the weaker hypotheses. In Section~\ref{sec:eqmeasdependscontinuously} we prove that the invariant measure depends continuously on the IFSm.
Finally, in Section~\ref{sec:examples} we provide some examples to illustrate the application of our results.

\section{IFSs on compact spaces with a compact set of maps}\label{sec:IFSs on compact spaces with a compact set of maps}

Consider a compact metric space $(X,d)$.
Let $C(X, \mathbb{R})$ be the set of continuous functions from $(X,d)$ to $\mathbb{R}$, and
${\rm Lip}_{1}(X)$ be the subset of $C(X, \mathbb{R})$ of Lipschitz functions with Lipschitz constant less or equal to $1$.

We recall that, for $\mathcal{P}(X)$ the set of regular probabilities over the Borel sigma algebra of  $(X, d)$, there exists a metric called the Hutchinson distance or Monge-Kantorovich (or  Wasserstein-1) distance defined by
$$d_{MK}(\mu,\nu)=\sup_{f \in {\rm Lip}_{1}(X)}\left\{ \int_{X}f d\mu - \int_{X}f d\nu\right\},$$
for any  $\mu, \nu \in \mathcal{P}(X)$. Since $(X,d)$ is compact, the topology induced by the metric is equivalent to the weak-* topology.

Now consider other compact set $(\Lambda, d_{\Lambda})$, which will be called the set of parameters. We will also consider the product space $(X \times \Lambda, d_{X \times \Lambda})$ where
$$ d_{X \times \Lambda}((x_1,\lambda_1),(x_2, \lambda_2))= \max\{d_X(x_1,x_2),d_{\Lambda}(\lambda_1, \lambda_2)\}.$$
The next definition appears in \cite{lewellen1993self} and \cite{mendivil1998generalization}. See also Dumitru~\cite{Dumitru2013AttractorsOI} for the Hutchinson-Barnsley theory for such infinite systems or Lukawska~\cite{lukawska_jachymski_2005} for infinite countable IFSs.
\begin{definition} \label{def:IFS}
  An IFS is a map $\tau:\Lambda \times X \to X$ which is continuous. Sometimes we denote the maps as $\tau(\lambda,x)=\tau_{\lambda}(x)$, to simplify the notation.
\end{definition}
\begin{notation}
  If there is no risk of confusion then we can omit the space of indices. We denote any IFS by $\mathcal{R}=(X, \tau)$ (compact notation),  where $\tau=\tau_{\lambda}$ and $\lambda \in \Lambda$.  Otherwise, we write explicitly $\mathcal{R}=(X, (\tau_{\lambda})_{\lambda\in \Lambda})$.
\end{notation}

Let $\mathcal{K}^{*}(X)$ be the set of nonempty compact parts of $X$.  If  $h$ is the Hausdorff distance, the fact that $X$ is compact implies that $(\mathcal{K}^{*}(X), h)$ is a compact metric space. (If $(X,d)$ was only complete, then $(\mathcal{K}^{*}(X), h)$ would be only complete.)
\begin{definition}
For each  IFS $\mathcal{R}=(X, \tau)$ we define the fractal operator by $F_{\mathcal{R}}: \mathcal{K}^{*}(X) \to \mathcal{K}^{*}(X)$ by
$$F_{\mathcal{R}}(B):= \tau(\Lambda \times B) =
\cup_{\lambda \in \Lambda} \tau_{\lambda}(B),$$
for any $B \in \mathcal{K}^{*}(X)$.
\end{definition}
The continuity of the map  $\tau:\Lambda \times X \to X$ and the compactness of $\Lambda$ imply that $F_{\mathcal{R}}(B)$ is indeed a compact subset of $X$, when $B$ is compact.

Let $\mathcal{P}(\Lambda)$ be the set of regular probabilities over the Borel sigma algebra of  $(\Lambda, d_{\Lambda})$.
We will also consider in $\mathcal{P}(\Lambda)$ the Monge-Kantorovich distance.
\begin{definition} \label{def:IFS with measures}
  An IFS with measures, IFSm for short, is an IFS $\tau:\Lambda \times X \to X$ endowed with a continuous family of measures $q: X \to \mathcal{P}(\Lambda)$.
\end{definition}
\begin{notation}
  If there is no risk of confusion then we can omit both the indices and the measure dependence at the points. We denote any IFSm by $\mathcal{R}=(X, \tau, q)$ (compact notation),  where $\lambda \in \Lambda$ and $q: X \to \mathcal{P}(\Lambda)$. Otherwise, we write explicitly $\mathcal{R}=(X, (\tau_{\lambda})_{\lambda\in \Lambda}, (q_x)_{x \in X})$.
\end{notation}
In the subsection \ref{sec:statpointofview} we will see how an IFSm defines naturally an interesting Markov process on $X$.
\begin{definition}\label{def:transfer operator}
   To each IFSm we assign a transfer operator  $B_{q}: C(X) \to C(X)$ defined by
   \begin{equation}\label{eq:transfer operator}
      B_{q}(f)(x):=\int_{\Lambda} f(\tau(\lambda,x)) dq_{x}(\lambda)
   \end{equation}
   for any $f \in C(X)$.
\end{definition}

By duality, $B_{q}$ induces a new operator in $\mathcal{P}(X)$:
\begin{definition}\label{def:Markov operator}
   To each IFSm we assign a Markov operator  $T_q:=B_{q}^*: \mathcal{P}(X) \to \mathcal{P}(X)$ defined by duality, for any $\mu \in \mathcal{P}(X)$ by
   \begin{equation}\label{eq:Markov operator}
     \int_{X} f(x) d T_{q}(\mu):=\int_{X}  B_{q}(f)(x) d\mu(x):=\int_{X} \int_{\Lambda} f(\tau(\lambda,x)) dq_{x}(\lambda) d\mu(x)
   \end{equation}
   for any $f \in C(X)$. A probability $\mu \in \mathcal{P}(X)$ is called invariant (with respect to the IFSm) if $T_{q}(\mu)=\mu$.
\end{definition}
An invariant measure is also called a  Hutchinson measure, which reflects its importance in the probabilistic or in the geometrical aspects of the theory.

Let $F_{\mathcal{R}}^{k}$ and $T_q^k$ denote respectively the $k-$th iterates of the Hutchinson and Markov operators.
The next theorem summarizes the expected results for a reasonable IFSm theory, called Hutchinson-Barnsley theory.

\begin{theorem}\label{th:main-tobeproved}
  Let $\mathcal{R}=(X, \tau, q)$ be a normalized IFSm,  under mild assumptions there exists a unique compact set $A_{\mathcal{R}}$ called fractal attractor and an unique probability $\mu_{\mathcal{R}}$ called invariant or Hutchinson measure, such that
\begin{enumerate}
  \item $F_{\mathcal{R}}(A_{\mathcal{R}})=A_{\mathcal{R}}$ and for any nonempty compact set $B \subseteq X$ we have $F_{\mathcal{R}}^{k}(B) \to A_{\mathcal{R}}$, w.r.t. the Hausdorff distance;
  \item $T_{q}(\mu_{\mathcal{R}})=\mu_{\mathcal{R}}$ and for any probability $\nu \subseteq \mathcal{P}(X)$ we have $T_{q}^{k}(\nu) \to \mu_{\mathcal{R}}$, w.r.t. the Monge-Kantorovich distance;
  \item $\supp(\mu_{\mathcal{R}})=A_{\mathcal{R}}$.
\end{enumerate}
\end{theorem}

In the next subsections, we will prove theorem \ref{th:main-tobeproved} and explain what exactly are the {\it mild conditions} of its statement
(see corollary \ref{cor:main-proved}).

\section{The statistical point of view}\label{sec:statpointofview}

Now we will define a stochastic process that help us understanding the importance of theorem \ref{th:main-tobeproved}.
Suppose we are given the IFSm $\mathcal{R}=(X, \tau, q)$.

\medskip We start by fixing $\mu_0 \in  \mathcal{P}(X)$;

\medskip Then we fix the initial position $Z_0 \in X$ according to the law

$Z_0 \sim \mu_0 \hspace{0.5cm} \mbox{(we randomly choose the initial position following the distribution $\mu_0$)};$

By recurrence, we fix, for any  $k \geq 0$,

$ \lambda_k \sim q_{Z_k}   \hspace{.5cm} \mbox{(we randomly choose the map to be used  in the $k+1$-th transition}$

	$\mbox{following the distribution $q_{Z_k}$),} $

$Z_{k+1}= \tau_{\lambda_k}(Z_k) \hspace{0.5cm} \mbox{(we calculate the k+1-th transition)}.$

\medskip

Now, let $\mu_n$ be the probability distribution of $Z_n$.
Under these conditions, we have two natural questions:

\medskip -
\textbf{Question 1:} is there $\mu_0 \in \mathcal{P}(X)$ which makes $\mu_n = \mu_0 $ for all $n\geq 1$ ? In other words, do there exists a stationary distribution $\mu_e$ for the process $Z_n$?

\medskip - \textbf{Question 2:} suppose the answer to question 1 is affirmative, i.e., there exists a stationary distribution  $\mu_e \in \mathcal{P}(X)$. Then, under which conditions can we prove that  $\mu_n \to \mu_e$ when $n \to \infty$, for any initial distribution $\mu_0 \in \mathcal{P}(X)$?

\medskip

The following result shows that the stationary distribution  $\mu_e$ is the fixed measure $\mu_{\mathcal{R}}$ given by the second item of theorem \ref{th:main-tobeproved} (also called invariant measure or Hutchinson measure).
\begin{theorem}\label{te:operadoratualizadistribuicao}
	If $\mu_k$ is the distribution of the random variable $Z_k$ defined above, then for all $k \geq 0$, we have the following relation between the distribution of $Z_k$ and $Z_{k+1}$:
	$$\mu_{k+1} = T_q(\mu_k).$$
\end{theorem}
\begin{proof}
	In order to characterize the distribution of $Z_{k+1}$ we just need to calculate the expectation of any continuous function of $Z_{k+1}$: We will do that by conditioning on $Z_k$: given $f \in C(X)$, we have
	$$E[f(Z_{k+1})] = E[E[f(Z_{k+1})|Z_k]]$$
	Now the distribution of $Z_k$ is given by $\mu_k$, and
	$$E[E[f(Z_{k+1})|Z_k]]= \int_X E[f(Z_{k+1})|Z_k=x] d\mu_k(x)$$
	Now, given that $Z_k=x$, we have that $Z_{k+1}=\tau_{\lambda}(x)$ where $\lambda$ is chosen according to the probability $q_x(\cdot)$.
	Therefore, we have
	$$E[f(Z_{k+1})|Z_k=x]= \int_{\Lambda}f(\tau_{\lambda}(x)) dq_x(\lambda).$$
	Combining the last equations we have
	$$E[f(Z_{k+1})] =  \int_X   \int_{\Lambda}f(\tau_{\lambda}(x)) dq_x(\lambda)  d\mu_k(x)= \int_X f(x) d T_q(\mu_k).$$
\end{proof}

Also, the following informal argument shows the importance of items 1 and 3 of theorem \ref{th:main-tobeproved}:
If $C_0=X$ e $C_{k+1}=F_{\mathcal{R}}(C_k)$ for all $k \geq 0$, the support of $\mu_k$ is contained in  $C_k$. Therefore, it is natural that the support of the stationary distribution $\mu_{\mathcal{R}}$  would be contained in the fractal attractor $A_{\mathcal{R}}$ (and in what follows we will prove in fact that they are equal, see theorem \ref{te:supp equal attrac}).

Finally, we give an interesting interpretation of the transfer operator in terms of conditional expectations of the stochastic process $\{Z_n\}$: if $f \in C(X)$, and $n \geq 0$, we have:
$$ B_q(f)(x) = E [f(Z_{n+1})|Z_n=x].$$

\section{Regularity hypotheses }\label{sec:hyp}

We will now present some hypotheses under which our results will be proved.
Let us fix an IFSm  $\mathcal{R} :=(X, \tau, q)$.
We begin by the following:

\begin{enumerate}
	\item[M1.] there exists $s >0$ such that, for any pair $x$ and $y$ in $X$, we have $$ \int_{\Lambda} \left| f(\tau(\lambda,x))-f(\tau(\lambda,y))\right|  dq_{x}(\lambda)  < s \; d(x,y)$$ for any map $f \in Lip_1(X)$;
	\item[H2.]  there exists $r\geq 0$, such that, the map $\tau(\cdot, x) \in {\rm Lip}_{r}(\Lambda)$, uniformly with respect to $x$;
	\item[H3.] there exists $t\geq 0$, such that,  the map $q: X \to \mathcal{P}(\Lambda)$ is in ${\rm Lip}_{t}(X)$;
	\item[H4.] $q_x(A)>0$ for any open subset $A \subseteq \Lambda$ and $x \in X$.
\end{enumerate}
Note that H3 can be stated as: for any fixed $x \in X$, the map
$$\begin{cases}
	\tau(\cdot, x): \Lambda \to X \\ \lambda \to \tau(\lambda, x)
\end{cases}$$
is $r$-Lipschitz.

Conditions M1, H2 and H4 are fairly general assumptions.
In the end of this section we will show three examples satisfying condition H3.
\begin{remark} Before going further, let us make a brief remark on the meaning of the word ``uniformly" as it appears in condition H2. When we say that 
	 the map $\tau(\cdot, x) \in {\rm Lip}_{r}(\Lambda)$, uniformly with respect to $X$, we mean that there exists $r \geq 0$ such that 
	$$d(\tau(\lambda_1,x),\tau(\lambda_2,x)) \leq  r d_{\Lambda}(\lambda_1,\lambda_2),$$ 
	for any pair  $\lambda_1$ and $\lambda_2$ in $\Lambda$ and any $x \in X$. Note that $r$ does not depend on $x$.  We will use the word "uniformly" also in conditions C1, W1, and CP1, with an analogous interpretation, meaning that the Lipschitz constant ($s$ or $1$) does not depend on $\lambda$ or $\lambda^N$.
\end{remark}  
Note that M1 holds for example if the following stronger assumption is assumed:
\begin{enumerate}
	\item[C1.]  the map $\tau(\lambda, \cdot) \in {\rm Lip}_{s}(X)$, with $s <1$, uniformly with respect to $\lambda$.
\end{enumerate}

On the other side, the following assumption is weaker that C1, and will be needed for some of our further results:
\begin{enumerate}
	\item[W1.]  the map $\tau(\lambda, \cdot) \in {\rm Lip}_{1}(X)$, uniformly with respect to $\lambda$.
\end{enumerate}

We now consider some assumptions weaker than C1 and M1, that will be enough for some of our results. We begin by denoting the  composition process as follows:
for any $\lambda^k=(\lambda_{0},\lambda_{1},\lambda_{2}, \ldots \lambda_{k-1}) \in \Lambda^{k}$, let
$$\tau_{\lambda^{k}}(x):=(\tau_{\lambda_{k-1}} \circ \cdots   \circ \tau_{\lambda_{0}})(x),$$
where $x\in X$ and $k \geq 1$, and $\tau_{\lambda^{0}}(x)=x$.
Sometimes we will use the notation $\tau(\lambda^{k},x)=\tau_{\lambda^{k}}(x)$.

Some of our results will remain valid if the uniform contraction holds for composition of maps:
We have the following assumption, weaker than C1:
\begin{enumerate}
	\item[CP1.]  there is $M \geq 1$ and $0<s <1$ such that the map
	$\tau_{\lambda^{M}}(\cdot) \in {\rm Lip}_{s}(X)$, uniformly with respect to $\lambda^M$.
\end{enumerate}

We also have the following assumption, weaker than M1:
\begin{enumerate}
	\item[MP1.] there exists $s >0$ and is $M \geq 1$ such that, for any pair $x$ and $y$ in $X$, we have $$ \int_{\Lambda^M} \left| f(\tau(\lambda^M,x))-f(\tau(\lambda^M,y))\right|
	d P^q_{M,x}(\lambda^M)
  < s d(x,y).$$
\end{enumerate}
where we define the measure $P_{M,x}^{q}$ on $\Lambda^{N}$
by
$$d P_{M,x}^{q}\left(\lambda_{0}, \ldots, \lambda_{M-1}\right)  =  dq_{\tau_{\lambda^{M-1}}(x)} (\lambda_{M-1}) \hdots
dq_{\tau_{\lambda_1}(\tau_{\lambda_0}(x))} (\lambda_2)
dq_{\tau_{\lambda_0}(x)} (\lambda_1)
dq_{x}(\lambda_0).$$
Note that, as C1 $\Longrightarrow$ M1, we also have  CP1 $\Longrightarrow$ MP1.

Also note that, with the notation just defined, the $M$-th iterate of the transfer operator is given by

\begin{equation}\label{eq:iterado-operador}
	B^M_q(f)(x)=  \int_{\Lambda^M}  f(\tau(\lambda^M,x))
	d P^q_{M,x}(\lambda^M).
\end{equation}

As a curiosity, we give an interpretation of the iterates of the transfer operator in terms of conditional expectations of the stochastic process $\{Z_n\}$ defined in Section~\ref{sec:statpointofview}: if $f \in C(X)$, $n \geq 0$, and $k \geq 1$, equation \eqref{eq:iterado-operador} help us to see that:
$$ B_q^k(f)(x) = E [f(Z_{n+k})|Z_n=x].$$

\section{Existence of the attractor}\label{sec:existence-of-the-attractor}
The next theorem can be founded in \cite[Theorem 3.2]{lewellen1993self} (also \cite{mendivil1998generalization}) but we provide a different and simpler proof.
\begin{theorem}\label{thm: unique invariant set}
		Consider an IFSm  $\mathcal{R} :=(X, \tau)$  satisfying C1.
		Then, $F_{\mathcal{R}}$ is $s$-Lipschitz contractive. In particular, there exists a unique invariant set $A_{\mathcal{R}}$.
\end{theorem}
\begin{proof}
	$$h(F_{\mathcal{R}}(B), F_{\mathcal{R}}(C))
	= h(\cup_{\lambda \in \Lambda} \tau_{\lambda}(B), \cup_{\lambda \in \Lambda} \tau_{\lambda}(C)) \leq  $$
	$$\sup_{\lambda} h(\tau_{\lambda}(B),\tau_{\lambda}(C)) \leq \sup_{\lambda} {\rm Lip}(\tau_{\lambda}) h(B,C)< s\; h(B,C),$$
	for any $B, C \in \mathcal{K}^{*}(X)$. The result follows from the Banach contraction theorem.
\end{proof}

An immediate consequence of Banach contraction theorem and Theorem~\ref{thm: unique invariant set} is the next result, which says the invariant set $A_{\mathcal{R}}$ is a topological attractor for the IFS (note that we have here a purely topological result, where the family of probabilities $q$ takes no role).

\begin{corollary}
	For any $B\in \mathcal{K}^{*}(X)$, we have $F_{\mathcal{R}}^{k}(B) \to A_{\mathcal{R}}$, w.r.t. the Hausdorff distance.
	More precisely,
	$$h(F_{\mathcal{R}}^{k}(B), A_{\mathcal{R}})< \frac{s^k }{1-s} h(F_{\mathcal{R}}(B), B),$$
	and we also have the following inequality, which in the classical case is stated as the Collage Theorem and corresponds to the case $k=0$:
	$$h(B, A_{\mathcal{R}})< \frac{h(F_{\mathcal{R}}(B), B)}{1-s}.$$
	Last inequality shows that, if we are able to choose $B$ such that
	$F_{\mathcal{R}}(B)$ is close to $B$, in the Hausdorff distance, than $B$ we will be close to the topological attractor.
\end{corollary}

\begin{remark}
	An  easy adaptation of the previous proof shows that the weaker hypothesis CP1 also implies the existence of an unique invariant set $A_{\mathcal{R}}$ which is also a topological attractor set, and given any nonempty compact set $B$, the iterates of $B$ under the Hutchinson operator converge exponentially to the attractor.
\end{remark}

\section{Existence of the invariant measure}\label{sec:Existence of the invariant measure}

From \cite{mendivil1998generalization} we know that in the case of IFSm where the family of measures $q: X \to \mathcal{P}(\Lambda)$ is constant, which means there exist an unique probability $p$ and $q_{x}=p \in \mathcal{P}(\Lambda)$ for all $x \in X$, than there exists an unique invariant probability $\mu_{\mathcal{R}}$ for the IFSm.

We will now extend this result in fairly greater generality:
\begin{theorem}\label{thm: unique invariant measure}
	Consider an IFSm  $\mathcal{R} :=(X, \tau, q)$  satisfying the  conditions M1, H2 and H3 above.
   Then $T_q$ is $s+ r \cdot t$-Lipschitz.
\end{theorem}

In order to prove theorem  \ref{thm: unique invariant measure}, we will need the following
\begin{lemma}\label{le:MarkovisLip}
Consider an IFSm  $\mathcal{R} :=(X, \tau, q)$  satisfying the  conditions M1, H2 and H3 above.	There exist $c>0$ such that, for any $f\in {\rm Lip}_{1}(X)$, we have $B_{q}(f) \in {\rm Lip}_{c}(X)$.
\end{lemma}
\begin{proof}
Let	$f\in {\rm Lip}_{1}(X)$.
Let us compare, for any $x, y \in X$ the difference
$$| B_{q}(f)(x) - B_{q}(f)(y)|= | \int_{\Lambda} f(\tau(\lambda,x)) dq_{x}(\lambda) - \int_{\Lambda} f(\tau(\lambda,y)) dq_{y}(\lambda)|= $$
$$= | \int_{\Lambda} f(\tau(\lambda,x))-f(\tau(\lambda,y)) + f(\tau(\lambda,y)) dq_{x}(\lambda) - \int_{\Lambda} f(\tau(\lambda,y)) dq_{y}(\lambda)|= $$
$$= | \int_{\Lambda} [f(\tau(\lambda,x))-f(\tau(\lambda,y))] dq_{x}(\lambda)  + \int_{\Lambda} f(\tau(\lambda,y)) dq_{x}(\lambda) - \int_{\Lambda} f(\tau(\lambda,y)) dq_{y}(\lambda)|\leq $$
$$\leq s  \cdot d(x,y) + r \cdot d_{MK}(q_{x},q_{y}) \leq (s+r \cdot t) d(x,y).$$
Now take $c=s+r \cdot t$.

\end{proof}

\begin{proof}[ of the Theorem~\ref{thm: unique invariant measure}]
	Consider $\mu, \nu \in \mathcal{P}(X)$. For any $f\in {\rm Lip}_{1}(X)$, lemma \ref{le:MarkovisLip} implies $c^{-1}B_{q}(f) \in {\rm Lip}_{1}(X)$. Then we have
   $$d_{MK}(T_{q}(\mu),T_{q}(\nu))=\sup_{f \in {\rm Lip}_{1}(X)} \int_{X}f dT_{q}(\mu) - \int_{X}f dT_{q}(\nu)= $$ $$= c \sup_{f \in {\rm Lip}_{1}(X)} \int_{X} c^{-1}B_{q}(f) d\mu - \int_{X} c^{-1}B_{q}(f) d\nu \leq c d_{MK}(\mu,\nu).$$
\end{proof}

  As a consequence of theorem \ref{thm: unique invariant measure} and the Banach contraction theorem, we have

\begin{corollary}\label{cor:existmedequil} Under conditions M1, H2, and H3, and supposing $s+ r \cdot t<1$, there exists a unique invariant probability measure $\mu_{\mathcal{R}}$. Also, for any regular probability measure $\mu_0$, we have $$T_{q}^{k}(\mu_0) \to \mu_{\mathcal{R}},$$ w.r.t. the Monge-Kantorovich distance. We also have exponential convergence under the Monge-Kantorovich distance.
\end{corollary}

Now we will prove that the weaker hypothesis MP1, together with W1, H2 and H3, are sufficient conditions for theorem \ref{thm: unique invariant measure} and  corollary \ref{cor:existmedequil}.

\begin{theorem}\label{thm:GEN unique invariant measure}
	Consider an IFSm  $\mathcal{R} :=(X, \tau, q)$  satisfying the  conditions W1, MP1, H2 and H3.
	Then $T^M_q$ is $s+ r \cdot t$-Lipschitz.
\end{theorem}

In order to prove theorem  \ref{thm:GEN unique invariant measure}, we will need the following:
\begin{lemma}\label{le:tauprodstillrLipshitz}
	Under Hypothesis W1 and H2,
	if we use in $\Lambda^k$ the distance given by
	$$d_{\Lambda^k}(\lambda^k,\tilde{\lambda}^k)= d_{\Lambda}(\lambda_0,\tilde{\lambda_0}) + \hdots + d_{\Lambda}(\lambda_{k-1},\tilde{\lambda}_{k-1}),$$
	where $\lambda^k=(\lambda_{0},\lambda_{1},\lambda_{2}, \ldots \lambda_{k-1}) \in \Lambda^{k}$ and	$\tilde\lambda^k=(\tilde\lambda_{0},\tilde\lambda_{1},\tilde\lambda_{2}, \ldots \tilde\lambda_{k-1}) \in \Lambda^{k}$,
	we have that, for any fixed $x \in X$, the map
	\begin{equation*}
	\begin{cases}
	\tau(\cdot, x): \Lambda^k \to X \\ \lambda^k \to \tau(\lambda^k, x)
	\end{cases}
	\end{equation*}
	is a $r$-Lipschitz map.
\end{lemma}
\begin{proof}
	We prove this lemma by induction on $k$: suppose it works for $k-1$,
	then
	$$
	d\Big((\tau_{\lambda_{k-1}} \circ \tau_{\lambda_{k-2}} \circ \cdots   \circ \tau_{\lambda_{0}})(x),
	(\tau_{\tilde\lambda_{k-1}} \circ \tau_{\tilde\lambda_{k-2}} \circ \cdots   \circ \tau_{\tilde\lambda_{0}})(x)\Big) \leq
	$$
	$$\leq
	d\Big((\tau_{\lambda_{k-1}} \circ \tau_{\lambda_{k-2}} \circ \cdots   \circ \tau_{\lambda_{0}})(x),
	(\tau_{\lambda_{k-1}} \circ \tau_{\tilde\lambda_{k-2}} \circ \cdots   \circ \tau_{\tilde\lambda_{0}})(x)\Big) + $$
	$$+d\Big((\tau_{\lambda_{k-1}} \circ \tau_{\tilde\lambda_{k-2}} \circ \cdots   \circ \tau_{\tilde\lambda_{0}})(x),
	(\tau_{\tilde\lambda_{k-1}} \circ \tau_{\tilde\lambda_{k-2}} \circ \cdots   \circ \tau_{\tilde\lambda_{0}})(x)\Big) \leq
	$$
	$$\leq
	d\Big(( \tau_{\lambda_{k-2}} \circ \cdots   \circ \tau_{\lambda_{0}})(x),
	(\tau_{\tilde\lambda_{k-2}} \circ \cdots   \circ \tau_{\tilde\lambda_{0}})(x)\Big) +  r d_{\Lambda}\Big(\lambda_{k-1} ,
	\tilde\lambda_{k-1} \Big) \leq
	$$
	$$\leq
	r d_{\Lambda^{k-1}}\Big((\lambda_{0},\lambda_{1},\lambda_{2}, \ldots \lambda_{k-2}),
	(\tilde\lambda_{0},\tilde\lambda_{1},\tilde\lambda_{2}, \ldots \tilde\lambda_{k-2})\Big)
	+  r d_{\Lambda}\Big(\lambda_{k-1} ,
	\tilde\lambda_{k-1} \Big) =
	$$
	$$=
	r d_{\Lambda^k}\Big((\lambda_{0},\lambda_{1},\lambda_{2}, \ldots \lambda_{k-1}),
	(\tilde\lambda_{0},\tilde\lambda_{1},\tilde\lambda_{2}, \ldots \tilde\lambda_{k-1})\Big). $$
	where the second inequality comes from W1 applied to the first distance and H2 to the second, and the third and last inequality comes from the induction hypothesis.
\end{proof}

\begin{lemma}\label{le:medprodMrLipshitz}
	Under hypotheses W1 and H3,	we have that, for any $k \geq 1$ and any pair $x$ and $y$ in $X$, $$d_{MK}(P^q_{k,x},P^q_{k,y}) \leq k t d(x,y).$$
\end{lemma}

\begin{proof}
	We prove this lemma by induction on $k$: suppose it works for $k-1$.
Remember that
	$$d P_{k,x}^{q}\left(\lambda_{0}, \ldots, \lambda_{k-1}\right)  =  dq_{\tau_{\lambda^{k-1}}(x)} (\lambda_{k-1})
	dq_{\tau_{\lambda^{k-2}}(x)} (\lambda_{k-2})\hdots
	dq_{\tau_{\lambda_1}(\tau_{\lambda_0}(x))} (\lambda_2)
	dq_{\tau_{\lambda_0}(x)} (\lambda_1)
	dq_{x}(\lambda_0)$$
	and define
	$$d P_{k,x,y}^{q}\left(\lambda_{0}, \ldots, \lambda_{k-1}\right)  =  dq_{\tau_{\lambda^{k-1}}(y)} (\lambda_{k-1})
	dq_{\tau_{\lambda^{k-2}}(x)} (\lambda_{k-2})\hdots
	dq_{\tau_{\lambda_1}(\tau_{\lambda_0}(x))} (\lambda_2)
	dq_{\tau_{\lambda_0}(x)} (\lambda_1)
	dq_{x}(\lambda_0).$$
	Note that
	$$d P_{k,x}^{q}\left(\lambda_{0}, \ldots, \lambda_{k-1}\right)  =  dq_{\tau_{\lambda^{k-1}}(x)} (\lambda_{k-1})
	d P_{k-1,x}^{q}\left(\lambda_{0}, \ldots, \lambda_{k-2}\right)$$
	while
	$$d P_{k,x,y}^{q}\left(\lambda_{0}, \ldots, \lambda_{k-1}\right)  =  dq_{\tau_{\lambda^{k-1}}(y)} (\lambda_{k-1})
	d P_{k-1,x}^{q}\left(\lambda_{0}, \ldots, \lambda_{k-2}\right).$$
	We will also need
	$$d P_{k,y}^{q}\left(\lambda_{0}, \ldots, \lambda_{k-1}\right)  =  dq_{\tau_{\lambda^{k-1}}(y)} (\lambda_{k-1})
	d P_{k-1,y}^{q}\left(\lambda_{0}, \ldots, \lambda_{k-2}\right).$$
	
	We know that
	$$d_{MK}(P^q_{k,x},P^q_{k,y}) = \sup_{g \in Lip^1(\Lambda^k)} \int_{\Lambda^k}g(\lambda^k) \Big(dP^q_{k,x}(\lambda^k)-dP^q_{k,y}(\lambda^k)\Big) $$
	Let $g:\Lambda^k \to \mathbb{R}$ be a $1$-Lipschitz map. We have
	$$
	\int_{\Lambda^k}g(\lambda^k) \Big(dP^q_{k,x}(\lambda^k)-dP^q_{k,y}(\lambda^k)\Big)=$$
	$$
	=\int_{\Lambda^k}g(\lambda^k) \Big(dP^q_{k,x}(\lambda^k)
	-dP^q_{k,x,y}(\lambda^k)+dP^q_{k,x,y}(\lambda^k)
	-dP^q_{k,y}(\lambda^k)\Big).
	$$
	We begin by controlling the first part of the integral above:
	$$\int_{\Lambda^k}g(\lambda^k) \Big(dP^q_{k,x}(\lambda^k)
	-dP^q_{k,x,y}(\lambda^k)\Big)
	= $$
	$$
	=\int_{\Lambda} \int_{\Lambda^{k-1}}
	g(\lambda^k)
	d P_{k-1,x}^{q}\left(\lambda_{0}, \ldots, \lambda_{k-2}\right)
	\Big(dq_{\tau_{\lambda^{k-1}}(x)} (\lambda_{k-1}) - dq_{\tau_{\lambda^{k-1}}(y)} (\lambda_{k-1})\Big)
	\leq $$
	$$ \leq d_{MK}(q_{\tau_{\lambda^{k-1}}(x)},q_{\tau_{\lambda^{k-1}}(y)})
	\leq t d(\tau_{\lambda^{k-1}}(x),\tau_{\lambda^{k-1}}(y))
	\leq t d(x,y)
	$$
	where we used W1 in the last inequality, H3 in the previous one, and the fact that
	$$\lambda_{k-1} \to \int_{\Lambda^{k-1}}
	g(\lambda^k)
	d P_{k-1,x}^{q}\left(\lambda_{0}, \ldots, \lambda_{k-2}\right)
	$$
	is 1-Lipschitz function.
	
	Now we control the second part of the integral:
	$$\int_{\Lambda^k}g(\lambda^k) \Big(dP^q_{k,x,y}(\lambda^k)
	-dP^q_{k,y}(\lambda^k)\Big)= $$
	$$	=\int_{\Lambda^{k-1}} \int_{\Lambda}
	g(\lambda^k)
	dq_{\tau_{\lambda^{k-1}}(y)} (\lambda_{k-1})
	\Big(d P_{k-1,x}^{q}\left(\lambda_{0}, \ldots, \lambda_{k-2}\right) - P_{k-1,y}^{q}\left(\lambda_{0}, \ldots, \lambda_{k-2}\right)\Big)\leq $$
	$$ \leq
	d_{MK}(P_{k-1,x}^{q},P_{k-1,y}^{q})
	\leq  (k-1) t  d(x,y)
	$$
	where we used the induction hypothesis, and the fact that
	$$\left(\lambda_{0}, \ldots, \lambda_{k-2}\right) \to \int_{\Lambda}
	g(\lambda^k)
	dq_{\tau_{\lambda^{k-1}}(y)} (\lambda_{k-1})
	$$
	is 1-Lipschitz function.	
\end{proof}

\begin{lemma}\label{le:GEN MarkovisLip}
	Consider an IFSm  $\mathcal{R} :=(X, \tau, q)$  satisfying the  conditions W1, MP1, H2 and H3.	There exist $c>0$ such that, for any $f\in {\rm Lip}_{1}(X)$, we have $B^M_{q}(f) \in {\rm Lip}_{c}(X)$.
\end{lemma}
\begin{proof}[ of the Lemma]
	Let	$f\in {\rm Lip}_{1}(X)$.
	Let us compare, for any $x, y \in X$ the difference
	$$| B^M_{q}(f)(x) - B^M_{q}(f)(y)|=
	 | \int_{\Lambda^M}  f(\tau(\lambda^M,x)) d P^q_{M,x}(\lambda^M)
	  -  \int_{\Lambda^M}  f(\tau(\lambda^M,y)) d P^q_{M,y}(\lambda^M) |= $$
	  $$ = 	 | \int_{\Lambda^M} \Big( f(\tau(\lambda^M,x))
	  - f(\tau(\lambda^M,y)) +f(\tau(\lambda^M,y)) \Big)
	   d P^q_{M,x}(\lambda^M)
	  -  \int_{\Lambda^M}  f(\tau(\lambda^M,y)) d P^q_{M,y}(\lambda^M) | \leq $$
		  $$ \leq  	 | \int_{\Lambda^M} \Big( f(\tau(\lambda^M,x))
	- f(\tau(\lambda^M,y)) \Big)
	d P^q_{M,x}(\lambda^M) |+ $$
	$$ + |\int_{\Lambda^M}  f(\tau(\lambda^M,y))
	d P^q_{M,x}(\lambda^M)
	-  \int_{\Lambda^M}  f(\tau(\lambda^M,y)) d P^q_{M,y}(\lambda^M) | = 	  $$

	$$\leq s  \cdot d(x,y) + r \cdot d_{MK}(P^q_{M,x},P^q_{M,y})
	 \leq  (s+r \cdot M \cdot t) d(x,y),$$
	where we used lemmas \ref{le:tauprodstillrLipshitz}
	and \ref{le:medprodMrLipshitz}.
	
Once the claim is proved, we take $c=s+r \cdot M \cdot t$ and the lemma is proved.

\end{proof}

\begin{proof}[ of the Theorem  \ref{thm:GEN unique invariant measure}]
	Consider $\mu, \nu \in \mathcal{P}(X)$. For any $f\in {\rm Lip}_{1}(X)$, lemma \ref{le:GEN MarkovisLip} implies $c^{-1}B^M_{q}(f) \in {\rm Lip}_{1}(X)$. Then we have
	$$d_{MK}(T^M_{q}(\mu),T^M_{q}(\nu))=\sup_{f \in {\rm Lip}_{1}(X)} \int_{X}f dT^M_{q}(\mu) - \int_{X}f dT^M_{q}(\nu)= $$ $$= c \sup_{f \in {\rm Lip}_{1}(X)} \int_{X} c^{-1}B^M_{q}(f) d\mu - \int_{X} c^{-1}B^M_{q}(f) d\nu \leq c d_{MK}(\mu,\nu).$$
\end{proof}

As a consequence of theorem \ref{thm:GEN unique invariant measure} and the Banach contraction theorem, we have

\begin{corollary}\label{cor:GEN existmedequil} Under conditions W1, MP1, H2, and H3,  and supposing $s+ r \cdot M  \cdot t<1$,
	there exists a unique invariant probability measure $\mu_{\mathcal{R}}$.
	Also, for any regular probability measure $\mu_0$, we have
	$$T_{q}^{k}(\mu_0) \to \mu_{\mathcal{R}},$$ w.r.t. the Monge-Kantorovich distance.
	We also have exponential convergence under such distance.

\end{corollary}
\begin{proof} It is a direct consequence of the following general fact, a generalization of the Banach  fixed point theorem:
	If $f:M\to M$ is a continuous map defined in a complete metric space $M$ such that some iterate $f^N$ is a $\lambda$-uniform contraction ($0\leq \lambda<1$), then we have an unique fixed point $p=f^N(p)$ for  $f^N$.
	Now, $f(p)=f(f^N(p))=f^N(f(p))$ which shows that $f(p)$ is also a fixed point for $f^N$. Therefore, $f(p)=p$ by uniqueness.
	To show that $p$ is attracting, given $k=Ni+j$,
	$$d(f^k(x),p)=d(f^k(x),f^k(p))= d(f^{Ni}(f^j(x)),f^{Ni}(f^j(p)))
	\leq  \lambda^i C $$
	where $$C=\max\{d(x,p), d(f(x),p), \hdots, d(f^{N-1}(x),p)  \}.$$
\end{proof}

\section{The support of the invariant measure coincides with the attractor}\label{sec:support}
The same conclusion in the next theorem holds under the stronger (but simpler)
hypotheses C1, H2, H3, H4. However we choose to prove that under less restrictive conditions. 
\begin{theorem}\label{te:supp equal attrac}
	Under hypothesis W1, CP1, H2, H3 and H4,
	we have  $$\supp(\mu_{\mathcal{R}})=A_{\mathcal{R}}.$$ 
\end{theorem}
\begin{proof}
	Remember that hypothesis CP1 implies hypothesis MP1, as well as C1 implies M1. Therefore, the hypotheses assumed here imply that 	$\mu_{\mathcal{R}}$ and $A_{\mathcal{R}}$ indeed exist.
	
	We begin by proving that $\supp(\mu_{\mathcal{R}}) \subseteq A_{\mathcal{R}}$, 	which is the same as proving that $x_0 \notin A_{\mathcal{R}}$ implies $x_0 \notin \supp(\mu_{\mathcal{R}})$.
	
	Suppose $x_0 \notin A_{\mathcal{R}}$. As the attractor is a closed set, there exists $\epsilon>0$ such that the ball centered in $x_0$ with radius $\epsilon$, denoted by $B(x_0,\epsilon)$, does not intersect $A_{\mathcal{R}}$:
	$$B(x_0,\epsilon) \cap A_{\mathcal{R}} = \emptyset.$$
	Now, there exist $K_0 \in \mathbb{N}$ such that whenever $k\geq K_0$ we have $h( F_{\mathcal{R}}^k(X), A_{\mathcal{R}}) < \frac{\epsilon}{2}$. Therefore, $k\geq K_0$ implies
	$$ \max_{z \in  F_{\mathcal{R}}^k(X)} d(z,A_{\mathcal{R}}) \leq h( F_{\mathcal{R}}^k(X), A_{\mathcal{R}}) < \frac{\epsilon}{2}$$
	and  we have
	$$ F_{\mathcal{R}}^k(X) \cap B(x_0, \epsilon/2) = \emptyset$$
	whenever $k\geq K_0$.
	Now let $f:X \rightarrow \mathbb{R}$ be any non-negative function
	that satisfy $f(z)=0$ if $z \notin B(x_0,\epsilon/2)$.
	We have
	$$\int f d\mu_{\mathcal{R}} = \int B_q^k(f) d\mu_{\mathcal{R}} =
	\int_X \int_{\Lambda^k}  f( \tau_{\lambda^{k}}(z)) d P^q_{k,x}(\lambda^k)d\mu_{\mathcal{R}} (z)=0 $$
	which equals zero because
	$	\tau_{\lambda^{k}}(z) \in F_{\mathcal{R}}^k(X)$ and
	$ F_{\mathcal{R}}^k(X) \cap B(x_0, \epsilon/2) = \emptyset$, which makes
	$f(  \tau_{\lambda^{k}}(z) )=0$ for any $z\in X$ and any $(\lambda_0, \lambda_1, \hdots, \lambda_{k-1})$ in the last integral.
		
	Now let us prove that $A_{\mathcal{R}} \subseteq \supp(\mu_{\mathcal{R}})$:
	Let $x_0 \in  A_{\mathcal{R}}$.
	Let $V$ be any open neighborhood of $x_0$. Let $\epsilon>0$ be such that the open ball $B(x_0,\epsilon) \subseteq V$.
	
	As a direct consequence of the definition of an attractor, there exist two sequences $(\alpha_1,\alpha_2,\hdots) \in \Lambda^{\mathbb{N}}$ and $(x_1,x_2,\hdots) \in  A_{\mathcal{R}}^{\mathbb{N}}$ which satisfy
	$$x_0 = 	\tau_{\alpha_1} \circ \tau_{\alpha_2} \circ \hdots \circ \tau_{\alpha_k} (x_k) \;\; \mbox{ for all }k \geq 0.$$
		
	Now  hypothesis CP1 (or C1) implies that
	$$\operatorname{diam} (\tau_{\lambda_1} \circ \tau_{\lambda_2} \circ \hdots \circ \tau_{\lambda_k}(X)) \to 0$$
	when $k \rightarrow +\infty$, for any sequence $\lambda_1,\lambda_2,\hdots \in  \Lambda^{\mathbb{N}}$.
	In particular there exists $N \geq 1$ such that
	$$\tau_{\alpha_1} \circ \tau_{\alpha_2} \circ \hdots \circ \tau_{\alpha_N}(X) \subseteq B(x_0, \epsilon/2).$$
   To prove this, let $N \geq 1$ be such that $
		diam (\tau_{\alpha_1} \circ \tau_{\alpha_2} \circ \hdots \circ \tau_{\alpha_N}(X))<\epsilon/2$.  Then for any $z \in X$, we have
		$d(\tau_{\alpha_1} \circ \tau_{\alpha_2} \circ \hdots \circ \tau_{\alpha_N}(z),x_0) = d(\tau_{\alpha_1} \circ \tau_{\alpha_2} \circ \hdots \circ \tau_{\alpha_N}(z),\tau_{\alpha_1} \circ \tau_{\alpha_2} \circ \hdots \circ \tau_{\alpha_N}(x_N)) \leq
		$
		$
		\leq
		diam (\tau_{\alpha_1} \circ \tau_{\alpha_2} \circ \hdots \circ \tau_{\alpha_N}(X)) < \epsilon/2.$

	Now we use hypothesis H2 (in fact we just need the continuity   of $\lambda \rightarrow \tau_{\lambda}(x)$) to conclude that
	there exists $\eta>0$ such that
	$$\tau_{\lambda_1} \circ \tau_{\lambda_2} \circ \hdots \circ \tau_{\lambda_N}(X)\subseteq B(x_0, \epsilon) $$
	whenever $d_{\Lambda}(\lambda_i,\alpha_i)<\eta$ for $1\leq i \leq N$.
	
	We now remember that $\mu_{\mathcal{R}}$ is the invariant distribution of the stochastic process $Z_n$ defined after theorem \ref{th:main-tobeproved} (provided the initial distribution of $Z_0$ is indeed $\mu_{\mathcal{R}}$).
	
	Let $$U = \Big\{ (\lambda_1,\lambda_2, \hdots, \lambda_N) \in \Lambda^N \;\|\; d_{\Lambda}(\lambda_i,\alpha_i)<\eta \;\;\forall\;\; 1\leq i \leq N\Big\}.$$
	
	Also note that the event $[Z_n \in V]$ contains the event $U$.
	Therefore we have $$\mu_{\mathcal{R}}(V)=P[Z_N \in V] \geq $$
	$$ \geq
	\int_X  \Big(   \int_{U}
	dq_{\tau_{\lambda^{N}}(x)} (\lambda_{N}) \hdots
	dq_{\tau_{\lambda^2}(x)} (\lambda_2)
	dq_{\tau_{\lambda^1}(x)} (\lambda_1)
	\Big) d\mu_{\mathcal{R}} (z) > 0, $$
	where the strict inequality is assured by H4.
	As the open neighborhood $V$ of $x_0$ is arbitrary, we have that $x_0 \in \supp(\mu_{\mathcal{R}})$.
\end{proof}

Now we can unify the  results above in the following corollary,
which says what the expression {\it mild assumptions} mean in theorem \ref{th:main-tobeproved}.

\begin{corollary}\label{cor:main-proved}
	  Let $\mathcal{R}=(X, \tau, q)$ be a  normalized IFSm, then under
	  conditions W1, CP1, H2, H3 and H4,   and supposing $s+ r \cdot M \cdot t<1$,	  all the conclusions of theorem \ref{th:main-tobeproved} are true.
\end{corollary}

\begin{remark}
	The same conclusion holds under the stronger (but simpler)
	hypotheses C1, H2, H3, H4, and supposing $s+ r \cdot t<1$,
\end{remark}

\section{The invariant measure depends continuously on the IFSp}\label{sec:eqmeasdependscontinuously}

 The next theorem is the analogous of \cite[Theorem 2]{mendivil1998generalization} for IFSm.  It shows that the Hutchinson measure depends continuously on the IFSm.

\begin{theorem}\label{thm: continuous dependence measure}
   Consider, for each $n \in \mathbb{N}$, an IFSm  $\mathcal{R}^{(n)} :=(X, \tau, q^{(n)})$  satisfying   conditions M1, H2, and H3,
    with $s+ rt<1$.
    Additionally assume the sequence $ q^{(n)}_{x} \to  q^{*}_{x}$ with respect to the distance $d_{MK}$ in $\mathcal{P}(\Lambda)$, uniformly in $x \in X$ (which means that, given any $\epsilon>0$ there exist a natural number $N$ such that $n\geq N$ implies  $d_{MK}\big(q_x^{(n)},q_x^{*}\big)<\epsilon$ for all $x\in M$).
    Let $\mu^{(n)}$ be the unique invariant measure for $\mathcal{R}^{(n)}$. Then there exists $\mu^{*}$ a unique invariant measure for $\mathcal{R}^{*} :=(X, \tau, q^{*})$, and $$\mu^{(n)}  \to  \mu^{*}$$ with respect to the distance $d_{MK}$ in $\mathcal{P}(X)$.
\end{theorem}
\begin{proof}
   First, we need to show that   $\mathcal{R}^{*}$ satisfies the same hypothesis of $\mathcal{R}^{(n)}$, i.e., M1,   H2, and H3, with $s+rt< 1$. Since the maps remains unchanged, we only need to show M1 and H3.
   But M1 is immediate, since Monge-Kantorovich distance is equivalent to weak*-topology, and $f$ and $\tau$ are continuous functions.
   In order to prove that $\mathcal{R}^{*}$ satisfies H3, we  need to show that the map $q^{*}: X \to \mathcal{P}(\Lambda)$ is in ${\rm Lip}_{t}(X)$.

    Remember the Monge-Kantorovich distance in $\mathcal{P}(\Lambda)$:
    $$d_{MK}(q_x^*,q_y^*)=\sup_{g \in {\rm Lip}_{1}(\Lambda)}\left\{ \int_{\Lambda } g d q_x^* - \int_{\Lambda} g d q_y^* \right\}.$$
         Take $g \in Lip_1(\Lambda)$ and any pair  $x,t \in X$, then for any  $n$ we have
   $$ \int_{\Lambda} g(\lambda) dq_{x}^{*}(\lambda) - \int_{\Lambda} g(\lambda) dq_{y}^{*}(\lambda) = \int_{\Lambda} g(\lambda) dq_{x}^{*}(\lambda) - \int_{\Lambda} g(\lambda) dq_{x}^{(n)}(\lambda) +$$
   $$ + \int_{\Lambda} g(\lambda) dq_{x}^{(n)}(\lambda) - \int_{\Lambda} g(\lambda) dq_{y}^{(n)}(\lambda) + \int_{\Lambda} g(\lambda) dq_{y}^{(n)}(\lambda) - \int_{\Lambda} g(\lambda) dq_{y}^{*}(\lambda) \leq$$
   $$\leq d_{MK}(q_{x}^{*}, q_{x}^{(n)})+ d_{MK}(q_{x}^{(n)}, q_{y}^{(n)})+ d_{MK}(q_{y}^{(n)}, q_{y}^{*}) \leq $$
    $$\leq d_{MK}(q_{x}^{*}, q_{x}^{(n)})+ t d(x, y)+ d_{MK}(q_{y}^{(n)}, q_{y}^{*}),$$
    where we used H3 for $\mathcal{R}^{(n)}$.
    As $d_{MK}(q_{x}^{*}, q_{x}^{(n)}) \to 0$ and
    $d_{MK}(q_{y}^{*}, q_{y}^{(n)}) \to 0$ for any $x$ and $y \in X$, we conclude that
    \begin{equation}\label{eq:IFSlimite}
    	d_{MK}(q_{x}^{*}, q_{y}^{*}) \leq t d(x, y),
    \end{equation}
    which means $\mathcal{R}^{*}$ also satisfy
     hypothesis H3.

	Therefore, as a consequence of corollary \ref{cor:existmedequil}, there exist $\mu^*$ the unique invariant probability measure for
	$\mathcal{R}^{*}$.

    Now we will prove that  $\mu^{(n)}  \to  \mu^{*}$ by proving that $d_{MK}(\mu^{(n)}, \mu^{*})$ converges to zero.
    We are now dealing with probabilities in $\mP(X)$, where
    $$d_{MK}(\mu^{(n)},\mu^*)=\sup_{f \in {\rm Lip}_{1}(X)}\left\{ \int_{X}f d\mu^{(n)} - \int_{X}f d\mu^*\right\}.$$
    We consider $f \in {\rm Lip}_{1}(X)$ and evaluate
    $$ \int_{X} f(x) d\mu^{(n)}(x) - \int_{X} f(x) d\mu^{*}(x).$$
    We recall that $T_{q^{(n)}}(\mu^{(n)})=\mu^{(n)}$ and $T_{q^{*}}(\mu^{*})=\mu^{*}$, thus
    $$ \int_{X} f(x) d\mu^{(n)}(x) - \int_{X} f(x) d\mu^{*}(x)= \int_{X} f(x) dT_{q^{(n)}}(\mu^{(n)})(x) - \int_{X} f(x) dT_{q^{*}}(\mu^{*})(x)=$$
    $$=\int_{X} B_{q^{(n)}}(f)(x) d\mu^{(n)}(x) - \int_{X} B_{q^{*}}(f)(x) d\mu^{*}(x) = $$
    $$= \int_{X} B_{q^{(n)}}(f)(x) d\mu^{(n)}(x) - \int_{X} B_{q^{*}}(f)(x) d\mu^{(n)}(x) + $$ $$+
    \int_{X} B_{q^{*}}(f)(x) d\mu^{(n)}(x) - \int_{X} B_{q^{*}}(f)(x) d\mu^{*}(x).$$

   Let us control the first part of the last sum above: H2 implies that $f(\tau(\cdot,x)) \in {\rm Lip}_{r}(\Lambda)$ and we write
   $$B_{q^{(n)}}(f)(x)  - B_{q^{*}}(f)(x)=  \int_{\Lambda} f(\tau(\lambda,x)) dq_{x}^{(n)}(\lambda) - \int_{\Lambda} f(\tau(\lambda,x)) dq_{x}^{*}(\lambda) \leq r d_{MK}(q_{x}^{(n)}, q_{x}^{*}).$$
   As, $q_{x}^{(n)} \to q_{x}^{*}$ uniformly in $x$ we can find, for each $\varepsilon>0$, $N_{\varepsilon} \in \mathbb{N}$ such that,
    $$d_{MK}(q_{x}^{(n)}, q_{x}^{*})< \varepsilon, \forall n \geq N_{\varepsilon}.$$
   Thus,
   $$\int_{X} B_{q^{(n)}}(f)(x) d\mu^{(n)}(x) - \int_{X} B_{q^{*}}(f)(x) d\mu^{(n)}(x) \leq r \varepsilon,  \;\;\; \forall n \geq N_{\varepsilon}.$$

   To control the second part of the sum, we claim that $B_{q^{*}}(f)$
   is a Lipschitz function, with Lipschitz function given by $rt+s$:
   $$B_{q^{*}}(f)(x) - B_{q^{*}}(f)(y)= \int_{\Lambda} f(\tau(\lambda,x)) dq_{x}^{*}(\lambda) - \int_{\Lambda} f(\tau(\lambda,y)) dq_{y}^{*}(\lambda) = $$
   $$ = \int_{\Lambda} f(\tau(\lambda,x)) dq_{x}^{*}(\lambda)
   -\int_{\Lambda} f(\tau(\lambda,x)) dq_{y}^{*}(\lambda) +$$
   $$ +  \int_{\Lambda} f(\tau(\lambda,x)) dq_{y}^{*}(\lambda)
    - \int_{\Lambda} f(\tau(\lambda,y)) dq_{y}^{*}(\lambda) \leq $$
   $$  \leq r. d_{MK}(q_{x}^{*}, q_{y}^{*}) + s d(x,y) \leq
   (rt+s) d(x,y),$$
	where we used H2, M1 and equation \eqref{eq:IFSlimite}.

   Thus,
   $$\int_{X} B_{q^{*}}(f)(x) d\mu^{(n)}(x) - \int_{X} B_{q^{*}}(f)(x) d\mu^{*}(x) \leq (rt+s) d_{MK}(\mu^{(n)}, \mu^{*}).$$

   Using this two inequalities we obtain,
   $$d_{MK}(\mu^{(n)}, \mu^{*}) \leq  r \varepsilon + (rt+s) d_{MK}(\mu^{(n)}, \mu^{*})$$
   $$d_{MK}(\mu^{(n)}, \mu^{*}) \leq \frac{r}{1- (rt+s)} \varepsilon,
   \;\;\;  \forall n \geq N_{\varepsilon}.$$
   From this, we conclude that $\mu^{(n)} \to \mu^{*}$ with respect to the distance $d_{MK}$.
\end{proof}

\section{Examples}\label{sec:examples}
In this section we explore some possibilities for IFS fulfilling different sets of hypothesis and use it to show the scope of our theorems.

We now show three examples where condition H3 is satisfied.
\begin{example}
	We can easily see that the constant case studied by \cite{mendivil1998generalization} and \cite{Arbieto2016WeaklyHyperbolic}, among many others, satisfies H3  for $t=0$.
\end{example}

\begin{example}	
	Now we will suppose condition C1, and consider the case studied in the thermodynamical formalism (see \cite{lopes2015entropy}, \cite{MO}, \cite{Brasil2022} and \cite{ElismarJairo24}) where $$dq_x(\lambda)= e^{A(\tau_{\lambda}(x))} d\mu(\lambda)$$  for a Lipschitz
	potential $A:X \to \mathbb{R}$  and a fixed probability $\mu \in \mathcal{P}(\Lambda)$
	satisfying $$x \to \int_{\Lambda} e^{A(\tau_{\lambda}(x))} d\mu(\lambda) \equiv {\rm cte},$$	
	(We could also suppose that $A$ is only H\"older, but here we will consider the Lipschitz case.)
	We claim that hypothesis H3 is satisfied if $$t={\rm diam}(\Lambda)\; e^{\|A\|_{\infty}}\; {\rm Lip}(A) \; s .$$	
	To prove this, we consider any $g \in {\rm Lip}_{1}(\Lambda)$,
	in order to estimate $d_{MK}(q_{x}, q_{y})$.
	
	Let $\lambda_{0}={\rm argmax}(g) \in \Lambda$. Define  $\tilde{g}(\lambda)= g(\lambda)-g(\lambda_{0}) \leq 0$. We have
	$$\int_{\Lambda} g(\lambda) dq_{x}(\lambda) - \int_{\Lambda} g(\lambda) dq_{y}(\lambda)=\int_{\Lambda} g(\lambda)-g(\lambda_{0}) dq_{x}(\lambda) - \int_{\Lambda} g(\lambda)-g(\lambda_{0}) dq_{y}(\lambda)= $$
	$$=\int_{\Lambda} \tilde{g}(\lambda)\left[ e^{A(\tau_{\lambda}(x))}- e^{A(\tau_{\lambda}(y))}\right] d\mu(\lambda).$$
	Notice that
	$${g}(\lambda)- {g}(\lambda_{0})\leq 1 \cdot d(\lambda, \lambda_0),$$
	thus $\tilde{g}(\lambda) \leq {\rm diam}(\Lambda)$. Using that, and the fact that
	$$e^{A(\tau_{\lambda}(x))}- e^{A(\tau_{\lambda}(y))} \leq
	e^{\|A\|_{\infty}} |A(\tau_{\lambda}(x))-A(\tau_{\lambda}(y))  |  \leq    $$
	$$ \leq e^{\|A\|_{\infty}} {\rm Lip}(A) |\tau_{\lambda}(x)-\tau_{\lambda}(y)  | \leq
	e^{\|A\|_{\infty}} {\rm Lip}(A)\; s \; d(x,y),$$
	we get
	$$\int_{\Lambda} g(\lambda) dq_{x}(\lambda) - \int_{\Lambda} g(\lambda) dq_{y}(\lambda) \leq \int_{\Lambda} {\rm diam}(\Lambda) e^{\|A\|_{\infty}} {\rm Lip}(A) \; s \;  d(x,y)  d\mu(\lambda)\leq $$ $$\leq {\rm diam}(\Lambda) e^{\|A\|_{\infty}} {\rm Lip}(A) \; s \;  d(x,y).$$
	Therefore, $$d_{MK}(q_x,q_y)=\sup_{g \in {\rm Lip}_{1}(\Lambda)}\left\{ \int_{\Lambda}g dq_x - \int_{\Lambda}g dq_y\right\} \leq $$
	$$\leq
	{\rm diam}(\Lambda) e^{\|A\|_{\infty}} {\rm Lip}(A) \; s \;  d(x,y).$$
\end{example}

\begin{example}\label{ex:h}
	We will consider now the case where
	\begin{itemize}
		\item[a)] $h:X \to [0,1]$ a Lipschitz a function, with Lipschitz constant $C>0$.
		\item[b)] two probability measures $\nu_1$ and $\nu_2$ on $\mathcal{P}(\Lambda)$ with full support.
		\item[c)] $q:X \to \mathcal{P}(\Lambda)$ given by $q_x= (1-h(x)) \nu_1 + h(x) \nu_2.$
	\end{itemize}
	
	We claim that hypothesis H3 is satisfied if $$t=C  d_{MK}(\nu_1,\nu_2).$$
	
	To prove this, we consider any $g \in {\rm Lip}_{1}(\Lambda)$,
	in order to estimate $d_{MK}(q_{x}, q_{y})$.
	
	We have
	$$\int_{\Lambda} g(\lambda) dq_{x}(\lambda) - \int_{\Lambda} g(\lambda) dq_{y}(\lambda)=$$
	$$= (1-h(x)) \int g d\nu_1 +h(x) \int g d\nu_2 -
	(1-h(y)) \int g d\nu_1 -h(y) \int g d\nu_2 =
	$$
	$$ = (h(y)-h(x)) \left(\int g d\nu_1  - \int g d\nu_2 \right).$$
	
	Therefore, $$d_{MK}(q_x,q_y)=\sup_{g \in {\rm Lip}_{1}(\Lambda)}\left\{ \int_{\Lambda}g dq_x - \int_{\Lambda}g dq_y\right\} \leq
	$$
	$$ \leq |h(y)-h(x)| \sup_{g \in {\rm Lip}_{1}(\Lambda)}\left\{ \int_{\Lambda}g d\nu_1 - \int_{\Lambda}g d \nu_2 \right\}
	\leq C  d_{MK}(\nu_1,\nu_2) d(x,y).$$
	
\end{example}

\begin{remark}
	We can easily prove that, in the Example~\ref{ex:h},   the Markov operator  $T_q:=B_{q}^*: \mathcal{P}(X) \to \mathcal{P}(X)$ can be written as
	$$T_q(\mu)=T_{\nu_1}((1-h) \mu ) + T_{\nu_2}(h \mu ).$$
	Note that $T_q$ is {\bf not} a constant linear combination of $T_{\nu_1}$ and $T_{\nu_2}$.
\end{remark}
An eventually contractive example obtained from a GIFS,  which satisfy conditions CP1 and MP1, is the following:
\begin{example}
	In 2008, \cite{Mihail2008} introduced the notion of generalized IFS, GIFS in short. The generalization was to consider maps from $X^m \to X$ in a metric space $(X,d)$ (called degree $m$ GIFS), instead a family of self maps in the classical IFS theory. Later, \cite{strobin2015attractors} show that  these GIFS has attractors that are not attractors of any contractive IFS. In the last years we can find several generalizations and applications of this ideas across the literature. An interesting case derived of that is Oliveira~\cite{Oliveira2017}, where an IFS is induced by a skew map originated from a GIFS. Those maps are not contractive, but they are eventually contractive leading to the existence a new class of attractors for GIFS. Inspired by that, we consider the following example.\\
	In this example we could have considered $\Lambda=\{1,\ldots, n\}$ finite, to focus on known theorems and to keep the attention at the eventual contraction rate. However we point out that in 2015 the study of GIFS was extended for an arbitrary set of maps by Dumitru~\cite{dumitru2015topological} (see also Miculescu~\cite{miculescu2018canonical}). We will consider $(X,d)$, a compact metric space, and a Lipschitz contractive GIFS of degree $m \geq 2$ denoted $S=(X, \phi_\lambda)_{\lambda \in \Lambda}$, where $\Lambda$ is a possibly infinite compact set. The family of maps $\phi_\lambda: X^m \to X$ satisfy\\
	$d_X(\phi_\lambda(x_1, \ldots , x_m), \phi_\lambda(y_1, \ldots , y_m)) \leq \sum_{i}^{m} a_{i}^{\lambda} d_X(x_i, y_i), \; 1 \leq i \leq m, \; \lambda \in \Lambda$. Thus $\phi_{\lambda}$ is a Lipschitz  contraction if $b_{\lambda}=\sum_{i}^{m} a_{i}^{\lambda} < 1, \forall \lambda \in \Lambda$, that is,
	$$d_X(\phi_\lambda(x), \phi_\lambda(y)) \leq b_{\lambda}  d_m(x, y).$$
	Additionally, we must require
	$s=\sup_{\lambda \in \Lambda} b_{\lambda} <1$ in order to obtain the power $m$ of induced IFS, $\mathcal{R}$, to be a contraction.
	
	Considering the maximum metric $$d_m((x_1, \ldots , x_m), (y_1, \ldots , y_m))= \max_{1\leq i \leq m}  d_X(x_i, y_i)$$ in $X^m$ we can see that
	$$d_X(\phi_\lambda(x_1, \ldots , x_m), \phi_\lambda(y_1, \ldots , y_m)) \leq \left(\sum_{i=1}^{m} a_{i}^{\lambda}\right)  d_m((x_1, \ldots , x_m), (y_1, \ldots , y_m)),$$ thus $\phi_\lambda$ is a Lipschitz  contraction if $b_j=\sum_{i=1}^{m} a_{i}^{\lambda} <1$, that is,
	$$d_X(\phi_\lambda(x), \phi_\lambda(y)) \leq b_\lambda  d_m(x, y).$$

	We now introduce the IFS given by $\mathcal{R}=(Z, \tau)$, where $(Z=X^m, d_m)$ is the complete metric space, $\Lambda$ is the same space as the GIFS and $\tau_{\lambda}: Z \to Z$ is given by
	$$\tau_{\lambda}(z) = (z_2, \ldots , z_{m-1}, \phi_{\lambda}(z)), $$
	for $\lambda \in \Lambda, \; \forall z=(z_1, \ldots , z_{m}) \in Z.$
	
	We notice that $\mathcal{R}$ is not a  Lipschitz contractive IFS, because for $$z=(z_1, z_2, \ldots, z_{m-1} , z_{m}),  z'=(z_1, z_2, \ldots, z_{m-1} , z'_{m}) \in Z$$ we have
	$d_m(z,z')=d_X(z_{m},z'_{m})$ and
	$$d_m(\tau_{\lambda}(z),\tau_{\lambda}(z'))= d_m( (z_2, \ldots, z_{m-1} , z_{m}, \phi_{\lambda}(z)), (z_2, \ldots, z_{m-1} , z'_{m}, \phi_{\lambda}(z')))=$$
	$$=\max \left[ d_X(z_{m},z'_{m}), d_X(\phi_{\lambda}(z),\phi_{\lambda}(z'))\right].$$
	
	From the contractivity of the GIFS we obtain
	$$d_X(\phi_{\lambda}(z), \phi_{\lambda}(z')) \leq \left(\sum_{i=1}^{m-1} a_{i}^{\lambda} \; d_X(z_i,z_i)\right) + a_{m}^{\lambda} \;d_X(z_m,z'_m) \leq  d_X(z_m,z'_m).$$
	Thus
	$$d_m(\tau_{\lambda}(z),\tau_{\lambda}(z'))=  d_X(z_{m},z'_{m}) = d_m(z,z').$$
	
	Consider $s = \max_{\lambda\in\Lambda} b_\lambda <1$ then  $\tau_{\lambda}(z)$ is a Lipschitz  contraction with contraction factor $s \in (0,1)$. Indeed, fixed $\lambda_1,\ldots, \lambda_m \in \Lambda$ we obtain
	$$d_m(\tau_{\lambda_m}(x),\tau_{\lambda_m}(x'))= d_m( (x_2, \ldots, x_{m-1} , x_{m}, \phi_{\lambda_m}(x)), (x'_2, \ldots, x'_{m-1} , x'_{m}, \phi_{\lambda_m}(x')))=$$
	$$=\max \left[ d_X(x_{2},x'_{2}), \ldots, d_X(x_{m},x'_{m}), d_X(\phi_{\lambda_m}(x),\phi_{\lambda_m}(x'))\right] \leq $$ $$\leq \max \left[ d_X(x_{2},x'_{2}), \ldots, d_X(x_{m},x'_{m}), b_{\lambda_m} d_m(x,x')\right].$$
	Taking\\
	$x=\tau_{\lambda_{m-1}}(y)=(y_2, \ldots, y_{m-1} , y_{m}, \phi_{\lambda_{m-1}}(y))$ and\\
	$x'=\tau_{\lambda_{m-1}}(y')=(y'_2, \ldots, y'_{m-1} , y'_{m}, \phi_{\lambda_{m-1}}(y'))$\\ we obtain
	$$d_m(\tau_{\lambda_m}(\tau_{\lambda_{m-1}}(y)),\tau_{\lambda_m}(\tau_{\lambda_{m-1}}(y'))) \leq $$ $$ \leq \max \left[ d_X(y_{3},y'_{3}), \ldots, d_X(\phi_{\lambda_{m-1}}(y),\phi_{\lambda_{m-1}}(y')), b_{\lambda_m} d_m(\tau_{\lambda_{m-1}}(y),\tau_{\lambda_{m-1}}(y'))\right] \leq $$
	$$\leq \max \left[ d_X(y_{3},y'_{3}), \ldots, b_{\lambda_{m-1}} d_m(y,y'), b_{\lambda_m} b_{\lambda_{m-1}} d_m(y,y') \right].$$
	
	After $m$ steps for $\tau_{\lambda^{m}} x =(\tau_{\lambda_{m}} \circ \cdots   \circ \tau_{\lambda_{1}})(x)$ we obtain a contraction factor of
	$$\max( b_{\lambda_1}, b_{\lambda_2} b_{\lambda_{1}}, \ldots, b_{\lambda_m} b_{\lambda_{m-1}} \cdots b_{\lambda_2} b_{\lambda_{1}}) \leq s.$$
	Thus $\mathcal{R}$ satisfy  CP1, there is $m \geq 1$ and $0<s <1$ such that the map 	$\tau_{\lambda^{m}}(\cdot) \in {\rm Lip}_{s}(X)$, uniformly with respect to $\lambda^m$.
\end{example}

\section*{Acknowledgements.}
Elismar R. Oliveira is partially supported by MATH-AMSUD under project GSA, financed by CAPES - process 88881.694479/2022-01, and by CNPq grant 408180/2023-4.

\end{document}